\documentclass[a4paper,10pt]{amsart}
\usepackage[utf8]{inputenc}
\usepackage{amsxtra}
\usepackage{amsopn}
\usepackage{amsmath,amsthm,amssymb}
\usepackage{amscd}
\usepackage{mathtools}
\usepackage{amsfonts}
\usepackage{latexsym}
\usepackage{verbatim}
\usepackage{MnSymbol}
\usepackage{xcolor}
\usepackage{cases}
\usepackage{tikz-cd}
\usepackage{cases}

\newtheorem{thm}{Theorem}[section]
\newtheorem*{thm*}{Theorem}
\newtheorem{prop}[thm]{Proposition}
\newtheorem{cor}[thm]{Corollary}
\newtheorem{lem}[thm]{Lemma}

\newtheorem{defn}[thm]{Definition}

\newcommand{\N}{\mathbb{N}}
\newcommand{\Z}{\mathbb{Z}}

\newcommand{\R}{\mathbb{R}}
\newcommand{\C}{\mathbb{C}}
\renewcommand{\H}{\mathcal{H}}

\let\c\overline

\newcommand{\diff}[1]{\frac{\partial}{\partial #1}}
\newcommand{\diffsq}[1]{\frac{\partial^2}{\partial #1^2}}

\newcommand{\delbar}{{\overline{\partial}}}
\newcommand{\del}{{{\partial}}}
\newcommand{\Vbar}{\overline{V}}
\newcommand{\phibar}{\overline{\phi}}
\newcommand{\mubar}{{\overline{\mu}}}
\newcommand{\im}{{\text{im}}}

\thanks{\newline 
The second author is partially supported by GNSAGA of INdAM}

\title[Almost complex structures on higher dimensional Kodaira-Thurston]{Left-invariant almost complex structures on the higher dimensional Kodaira-Thurston manifolds}
\author{Tom Holt and Riccardo Piovani}
\begin{document}
\begin{abstract}
We develop computational techniques which allow us to calculate the Kodaira dimension as well as the dimension of spaces of Dolbeault harmonic forms for left-invariant almost complex structures on the generalised Kodaira-Thurston manifolds.
\end{abstract}
\maketitle

\section{Introduction}
Let $(M,J)$ be an almost complex manifold  of real dimension $2n$; then the exterior derivative on $(p,q)$-forms decomposes as
\[
d=\mu+\del+\delbar+\mubar:A^{p,q}\to A^{p+2,q-1}\oplus A^{p+1,q}\oplus A^{p,q+1}\oplus A^{p-1,q+2}.
\]
According to the celebrated Theorem of Newlander and Nirenberg, $J$ is integrable, \textit{i.e.}, it is induced by the structure of a complex manifold, if and only if $\mu=0$, or equivalently $\delbar^2=0$.

Fixing an almost Hermitian metric $g$ and denoting by $*:A^{p,q}\to A^{n-q,n-p}$ the $\C$-linear Hodge star operator, the $L^2$-formal adjoint of $\delbar$ is given by $\delbar^*:=-*\del*$. Therefore one can define the {\em Dolbeault Laplacian}, as in the integrable case,  as the second order differential operator acting on the space $A^{p,q}$ of $(p,q)$-forms on $(M,J)$ as 
\[
\Delta_\delbar:=\delbar\delbar^*+\delbar^*\delbar:A^{p,q}\to A^{p,q}.
\]
It turns out that  $\Delta_\delbar$ is elliptic and formally self adjoint. By elliptic theory, if $M$ is compact, then the space of Dolbeault harmonic $(p,q)$-forms $\H^{p,q}_\delbar:=\ker\Delta_\delbar\cap A^{p,q}$ has finite dimension $h^{p,q}_\delbar$, and if moreover $J$ is integrable, then $\H^{p,q}_\delbar$ is isomorphic to the Dolbeault cohomology group $H^{p,q}_\delbar:=\frac{\ker\delbar}{\im\delbar}$, which is an invariant of the complex structure.

In Problem 20 of Hirzebruch's 1954 Problem List \cite{Hi}, Kodaira and Spencer asked if the number $h^{p,q}_\delbar$ is independent on the choice of the almost Hermitian metric $g$. The first author and Zhang just recently gave a negative answer to this problem \cite{HZ,HZ2}, building a family of almost complex structures on the 4-dimensional Kodaira-Thurston manifold, namely the product between the circle $S^1$ and the compact quotient  $H_{3}(\Z)\backslash H_{3}(\R)$ of the Heisenberg group,  showing that $h^{0,1}_\delbar$ varies with different choices of the almost Hermitian metric. We refer to  \cite{Ho,HP,P,PT5,TT} for further studies of Dolbeault harmonic forms on almost Hermitian manifolds.
We remark that, at the current state of the art, the only known example of a non integrable almost Hermitian structure where it is possible to compute $h^{p,q}_\delbar$ completely for all $0\le p,q\le n$ is just the 4-dimensional Kodaira-Thurston manifold. Furthermore, to our knowledge there are no known higher dimensional examples where one can actually compute $h^{p,q}_\delbar$ except for the $q=0$ case, where the computation does not involve the metric \cite{TT6}.

Recently in \cite{CZ,CZ2} Chen and Zhang defined the {\em Kodaira dimension} for almost complex manifolds, extending the well known notion for complex manifolds. We point out that all the known examples of compact solvmanifolds endowed with a left-invariant almost complex structure have Kodaira dimension 0 or $-\infty$ (see \cite{CZ,CNT,CNT2, CNT3}).

The aim of this paper is to develop computational techniques to calculate the Kodaira dimension $\kappa_J$ and the numbers $h^{p,q}_\delbar$ on the {\em generalised Kodaira-Thurston manifolds}
$$
KT^{2n+2}=H_{2n+1}(\Z)\backslash H_{2n+1}(\R)\times S^1,
$$
where $H_{2n+1}(\R)$ denotes the generalised Heisenberg group (see, \textit{e.g.}, \cite[p.24]{BT}), endowed with natural almost complex and almost Hermitian structures. Note that for $n=1$, $KT^4$ is the $4$-dimensional Kodaira-Thurston manifold.

In Section 4 we prove the following (see Theorem \ref{kodaira-dimension-KT})
\begin{thm*}
 On the $(2n+2)$-dimensional Kodaira-Thurston manifold $KT^{2n+2}$ endowed with a left-invariant almost complex structure $J$, there are only two possible values for the Kodaira dimension: $0$ or $-\infty$.
\end{thm*}
We are able to compute the Kodaira dimension thanks to the characterisation of $L^2$ functions on $KT^{2n+2}$ via the regular representation of the Heisenberg group $H_{2n+1}(\R)$ (see \cite{Aus}).

In Sections 6 and 7, we calculate $h^{0,1}_\delbar$ for a family of left-invariant almost Hermitian structures on $KT^4$ and $KT^{6}$, respectively.  
To do so, we make use of a description of the eigenfunctions of the Hodge-de Rham Laplacian on $KT^{2n+2}$ using Hermite polynomials (see \cite{DS}). Note that this method has technical differences with the one used in \cite{HZ} on $KT^4$ to answer the Kodaira-Spencer question. 
In fact, in Section 6 our technique allows the same values of $h^{0,1}_{\delbar}$ to be computed on the same almost Hermitian structure on $KT^4$ in a much simpler way. In Section 7 we prove that the space $\H^{0,1}_\delbar$ on $KT^{6}$, endowed with a family of left invariant almost complex structures and any compatible left invariant almost Hermitian metric, consists of left-invariant forms.

We refer to \cite{MPR} for another study of the eigenfunctions of the Hodge-de Rham Laplacian on the generalised Heisenberg group, and to \cite{MPR2} for the issue of Hodge decomposition for the Hodge Laplacian on functions and on differential forms on the same manifold.

\medskip\medskip
\noindent{\em Acknowledgments.}
We would like to thank Adriano Tomassini for his suggestions which provided the idea for this paper and also for several fruitful discussions. Additionally, we thank Fulvio Ricci for providing some good references and helpful comments. 

\section{Representation Theory}

Let $X$ be a smooth homogeneous space, in other words, for some Lie group $G$, $X$ is given by the space of cosets of a closed Lie subgroup $\Gamma \subset G$
$$X = \Gamma \backslash G := \left\{\Gamma g \in G\, \middle|\, g \in G \right\}. $$
The \textit{(right) regular representation} of $G$ is defined on $L^2(X)$ with the right Haar measure by
\begin{align*}
    R(g):L^2 (G) &\rightarrow L^2 (G)\\
    f&\mapsto f \circ r_g
\end{align*}
where $r_g$ is the right multiplication function, \textit{i.e.}, $r_g: h \mapsto hg$ for all $g,h \in G$.
Note that $r_g$ is also well-defined as a map on $X$, sending left cosets to left cosets, 
$$r_g:\Gamma h \mapsto \Gamma hg.$$ 
We shall assume that the right Haar measure coincides with the left Haar measure, \textit{i.e.}, $G$ is unimodular. In this case the Haar measure descends to a measure on $X$ and thus the regular representation on $L^2(G)$ induces a unitary representation on $L^2(X)$
\begin{align*}
    R(g):L^2 (X) &\rightarrow L^2 (X)\\
    f&\mapsto f \circ r_g
\end{align*}
which we will also call the regular representation. Nilpotent groups provide an excellent family of groups which are all unimodular \cite[Corollary of Prop. 25]{Nac}. Indeed the Kodaira-Thurston manifolds we shall consider later are all given by the quotient of some nilpotent groups.

A decomposition of $L^2(X)$ into closed subspaces $\mathcal{S}_I$,
\begin{equation}\label{decomp}
    L^2(X) = \widehat{\bigoplus_{I\in \mathcal I}} \mathcal{S}_I,
\end{equation}
is called a decomposition of the regular representation if each of the spaces $\mathcal{S}_I$ is preserved by $R(g)$ for all $g \in G$. Here $\hat{\oplus}$ denotes the closure of the direct sum.

Let $\mathfrak{g}$ be the Lie algebra of $G$ with a basis given by the vectors $\nu_1, \dots, \nu_n \in \mathfrak g$. Extending $\nu_j$ left-invariantly to the whole of $G$, there is a well-defined pushforward to a vector field on $X$, which we also call $\nu_j$. We will call any vector field on $X$ defined in this way \textit{left-invariant}. Interpreting $\nu_j$ as a directional derivative on $X$, we say that a linear differential operator $P$ on $X$ is \textit{left-invariant} if it can be written in the form
$$
P = \sum_{\alpha \in \mathbb{N}^n,\,|\alpha|\le k } c_\alpha D^{\alpha} 
$$
where $\alpha= (\alpha_1, \dots, \alpha_n)$ is a multi-index, $|\alpha|=\alpha_1+\dots+\alpha_n$, $c_\alpha\in\C$ is a family of constants, and we define $D^{\alpha}$ to be the differential operator $\nu_1^{\alpha_1}\nu_{2}^{\alpha_2} \dots \nu_n^{\alpha_n}$.

Any decomposition of the regular representation \eqref{decomp} also gives us a  decomposition of all left-invariant differential operators. This should not be surprising as the directional derivative given by any $\nu \in \mathfrak g$ is also given by the differential of the regular representation
\begin{align*}
    dR(\nu): L^2(X) \cap C^{\infty}(X) &\rightarrow L^2(X) \cap C^{\infty}(X)\\
    f&\mapsto \nu f.
\end{align*}

For the convenience of the reader, we provide a simple proof of the relevant results below. For a more general discussion of the properties of the differential of a representation, we direct the reader to \cite{Bal}. 
\begin{thm}\label{S preserved}
    If $\mathcal{S}$ is a closed subspace of $L^2(X)$ preserved by the right regular representation then $\mathcal{S} \cap C^{\infty}$ is preserved by left-invariant vector fields.
    \begin{proof}
        An element of the Lie algebra $\nu \in \mathfrak g$ can be given as the tangent to some curve $\gamma(t)$ on $G$, passing through the identity at $t = 0$. The corresponding vector field on $X$ is then given at a general point $\Gamma g \in X$ as the tangent to the curve $\Gamma g \gamma(t)$. 

        Applying $\nu$ to a smooth function $f \in \mathcal{S} \cap C^{\infty}$ gives us the directional derivative
     \begin{align*}
         \nu f(\Gamma g) &= \left.\frac{d}{dt}\right\rvert_{t=0} f(\Gamma g\gamma(t))\\
         &= \left.\frac{d}{dt}\right\rvert_{t=0} R(\gamma(t))f(\Gamma g)\\
         &= \lim_{t \rightarrow 0}\frac{R(\gamma(t))f(\Gamma g) - R(\gamma(0))f(\Gamma g)}{t}.
     \end{align*}
     Since $\mathcal{S}$ is preserved by $R$, we know have
      $$\frac{R(\gamma(t))f(g) - R(\gamma(0))f(g)}{t} \in \mathcal{S} $$
      for all $t >0$.
      But $\mathcal{S}$ is closed and so we must have $Vf \in \mathcal{S}$. 
    \end{proof}

\end{thm}

\begin{cor}\label{Cor decomp of solns}
    Let $P$ be a left-invariant linear differential operator on $X = \Gamma \backslash G$, and let $f \in C^{\infty}(X)$ be a smooth function which can be decomposed, with respect to the regular representation, into the sum
    $$f = \sum_{I\in \mathcal{I}} f_I, $$
    with $f_I \in \mathcal{S}_I \cap C^{\infty}$.

    If $f$ is a solution to the differential equation
    $$Pf = 0 $$
    then the functions $f_I$ are also solutions, for all $I \in \mathcal{I}$.
    \begin{proof}
        Decomposing the equation $Pf = 0$ we get
        $$P f = \sum_{I \in \mathcal{I}} P f_I = 0. $$
        Since the spaces $\mathcal{S}_I$ are all closed subspaces of $L^{2}(X)$, by Theorem \ref{S preserved} we can say that $P f_I \in S_I$. Furthermore, the spaces $\mathcal{S}_I$ are all mutually orthogonal and so $P f = 0$ if and only if each of its components are zero, \textit{i.e.} $P f_I = 0$.
    \end{proof}
\end{cor}

\section{The Kodaira-Thurston manifold}
The $2n+2$-dimensional Kodaira-Thurston manifold $KT^{2n+2}$ is a homogeneous space given by
$$ KT^{2n+2} = \left( H_{2n+1}(\Z) \times \mathbb{Z} \right)\backslash \left(  H_{2n+1}(\R) \times \mathbb{R}\right) $$
where $H_{2n+1}(\mathbb{R})$ denotes the Heisenberg group 
$$ H_{2n+1}(\mathbb{R}) = \left\{\begin{pmatrix}    
    1 & x & z\\
    0_n & \text{\Large{$I_n$}} & y\\
    0 &  0_n & 1
\end{pmatrix}\,\middle|\, x,y \in \mathbb{R}^n,\,  z \in \mathbb{R} \right\} $$
and $H_{2n+1}(\mathbb{Z})$ denotes the discrete subgroup, with $x,y \in \mathbb{Z}^n$, $z \in \mathbb{Z}$. Here $I_n$ is the $n\times n$ identity and $0_n$ is the zero vector $(0,\dots, 0) \in \mathbb{R}^n$.

The tangent bundle of $KT^{2n+2}$ is spanned at every point by the left-invariant vectors
$$\nu_1 = \diff {x_1}, \dots, \nu_n = \diff {x_n},$$
$$ \nu_{n+1} = \diff {y_1} + x_1 \diff z, \dots, \nu_{2n} = \diff{y_{2n}} + x_{2n} \diff z,$$
$$ \nu_{2n+1} = \diff z,\  \nu_{2n+2} = \diff t,  $$
where the variables $x,y,z$ parametrise the Heisenberg group $H_{2n+1}({\R})$ as described above, and $t$ parametrises $\R$. 

The cotangent bundle of $KT^{2n+2}$ is spanned at every point by the dual left-invariant $1$-forms
$$e^1 = d{x_1}, \dots, e^n = d{x_n},$$
$$f^1 = d{y_1}, \dots, f^n = d{y_{2n}},$$
$$e^{n+1} = dz-x_1dy_1-\dots-x_ndy_n,\  f^{n+1} = dt,$$
where the only non zero structure equation is
\[
de^{n+1}=-e^1\wedge f^1-\dots -e^n\wedge f^n. 
\]

The regular representation of $H_{2n+1}\times \R$ on $L^2(KT^{2n+2})$ has an irreducible decomposition, described by
\begin{equation}\label{decomp of KT6}
    L^2(KT^{2n+2}) = \widehat{\bigoplus_{I \in \mathcal{I}}}\mathcal{S}_I \oplus \widehat{\bigoplus_{J\in \mathcal{J}}}\mathcal{T}_{J} 
\end{equation}
where the the first sum is taken over all $I = (p,q,l)$ such that $ p,q \in \Z^n, l\in \Z$ and the second sum is taken over all $J = (q,m,l)$ such that $ m \in \Z\backslash \{0\}, q \in (\Z/m)^n, l\in\mathbb{Z}$. In \eqref{decomp of KT6} $\mathcal{S}_{I}$ is a 1 dimensional space spanned by a single function 
$$ \mathcal{S}_{I} = \mathbb{C} \cdot \left\{ e^{2\pi i(p\cdot x + q \cdot y+lt)}\right\} $$
and $\mathcal{T}_{J}$ is an infinite dimensional space given by
$$\mathcal{T}_{J} = \left\{ e^{2\pi i (q\cdot y + mz +lt)}\sum_{\xi \in \Z^n}\psi(x + \xi) e^{2\pi i m \xi \cdot y} \,|\, \psi \in L^2(\R^n)\right\}. $$
See \cite[Section I.5]{Aus} for a decomposition of the regular representation of $H_{2n+1}(\mathbb R)$ acting on $L^2(H_{2n+1}(\mathbb Z) \backslash H_{2n+1}(\mathbb{R}))$, from which the above decomposition is easily obtained.

We define the map
\begin{align*}
    W_J: L^2(\mathbb{R}^n) \rightarrow \mathcal{T}_J,
\end{align*}
in the natural way. 
This is a generalised version of the Weil-Brezin map on the Heisenberg manifold.
Using $W_J \psi$ to denote a general element of $\mathcal{T}_J$, if we have $W_J \psi\in C^{\infty}(KT^{2n+2})$ then by classical Fourier analysis $\psi \in C^{\infty}(\mathbb{R}^n)$ and the sum over $\xi \in \mathbb{Z}^n$ is absolutely convergent. We can then write
\begin{align*}
    \diff{x_j} W_J \psi &= W_J \left(\diff{x_j} \psi\right)\\
    \left(\diff{y_j} + x_j \diff{z}\right) W_J \psi &= 2\pi i  q_j W_J \psi + 2\pi im W_J (x_j \psi)\\
    \diff z W_J \psi &= 2\pi i m W_J \psi\\
    \diff t  W_J \psi&= 2 \pi i l W_J \psi.
\end{align*}
This shows us that $W_J \psi$ is smooth if and only if the Weil-Brezin map 
$$W_J \left( x^\alpha \diff{x}^\beta\psi \right) $$
converges for all multi-indexes $\alpha,\beta \in \mathbb{N}^n$, where $x^\alpha := x_1^{\alpha_1} x_2^{\alpha_2}\dots x_n^{\alpha_n}$ and similarly $\diff{x}^\beta := \left(\diff{x_1}\right)^{\beta_1} \left(\diff{x_2}\right)^{\beta_2} \dots \left(\diff{x_n}\right)^{\beta_n}$. In other words, we must have $\psi \in \mathcal{S}(\mathbb{R}^n)$, where 
$$S(\R^n) := \left\{f \in C^{\infty}(\mathbb{R}^n) \,\middle|\, \forall \alpha, \beta \in \N^n  \sup_{x \in \mathbb{R}^n} \left\| x^\alpha \diff{x}^{\beta} f \right\| < \infty \right\} $$
is the space of Schwartz functions.

We can therefore write
$$\mathcal{T}_J \cap C^{\infty}(KT^{2n+2}) = \left\{ W_J \psi \,|\, \psi \in S(\R^n)\right\} $$
when restricting to the space of smooth functions.

Conversely, the space $\mathcal{S}_I$ is unchanged when restricting to smooth functions, \textit{i.e.}, we have $\mathcal{S}_I \cap C^{\infty}(KT^{2n+2}) = \mathcal{S}_I$.

\section{The Kodaira Dimension of $KT^{2n+2}$}

In \cite{CZ}, Chen and Zhang introduced the notion of Kodaira dimension for almost complex manifolds. We now recall their definition. Given a $2n$-dimensional almost complex manifold $(X,J)$, denote the \textit{canonical bundle}, \textit{i.e.}, the bundle of $(n,0)$-forms, by $K_J$. We define the $Nth$-\textit{plurigenus} of an almost complex manifold $(X,J)$, as in \cite[Definition 1.2]{CZ}, to be the space of $\delbar$-closed sections of $K_J^{\otimes N}$ $$H^0(X,K_J^{\otimes N}) = \left\{ s \in K_J^{\otimes N} \,\middle| \,\delbar s = 0 \right\}.$$ 
Then, denoting
$$P_N = \dim H^0(X,K_J^{\otimes N}) $$
the Kodaira dimension of $(X,J)$ is defined to be
$$\kappa_J(X) = \begin{cases}
    -\infty & P_N = 0,\,\, \forall N \geq 0,\\
    \limsup_{N\rightarrow \infty} \frac{\log P_N}{\log N} & \text{otherwise.}
    
\end{cases} $$

We will now demonstrate how it is possible to calculate the Kodaira dimension of the Kodaira-Thurston manifold in any dimension and given any left-invariant almost complex structure. This same argument could be applied to calculate the Kodaira dimension for other solvmanifolds with left-invariant almost complex structure.   

\begin{thm}\label{kodaira-dimension-KT}
    On the $(2n+2)$-dimensional Kodaira-Thurston manifold $KT^{2n+2}$ endowed with a left-invariant almost complex structure $J$, there are only two possible values for the Kodaira dimension: $0$ or $-\infty$.
\end{thm}
\begin{proof}
For a general left-invariant almost complex structure, we can assume that the space of $T^{1,0}KT^{2n+2}$ is spanned at each point by the $(1,0)$-vectors $V_1, \dots, V_{n+1}$ given by
$$V_j = \sum_{i = 1}^{2n+2} A_{ij} \nu_i $$
for some choice of $A_{ij} \in \mathbb{C}$. Let $\phi^j$ be the dual $(1,0)$-forms, then in order to find the plurigenera, we must find $f \in C^{\infty}(KT^{2n+2})$ such that
$$\delbar \left(f \left( \phi^{1 2 \dots n+1}\right)^{\otimes N}\right) = \left( \sum_{j} \Vbar_j f  \phibar^j\right)\otimes \left(\phi^{1 2 \dots n+1}\right)^{\otimes N} + f \delbar \left(\phi^{1 2 \dots n+1}\right)^{\otimes N}= 0. $$
Maintaining full generality, we can assume that
$$ \delbar \phi^{1 2 \dots n+1} = \left(\sum_{j=1}^{n+1} C_j \phi^{\overline{j}}\right) \wedge \phi^{12 \dots n+1},  $$
where $C_j \in \mathbb{C}$ are constant.
The equations to solve are therefore
\begin{equation}\label{Kod dim eqns}
    \Vbar_j f + N C_j f  = 0. 
\end{equation}

The operators $\Vbar_j + NC_j$ are all left-invariant and so, if we find the solutions when $f \in \mathcal{S}_I$ or $f \in \mathcal{T}_J$, for some $I \in \mathcal{I}$ and $J\in \mathcal{J}$, all other solutions can be found through linear combinations of these. In the two lemmas below, we consider first the solutions in $\mathcal{S}_I$ followed by the solutions in $\mathcal{T}_J$.

\begin{lem}\label{solns in S}
    If there exists a solution to \eqref{Kod dim eqns} in $\mathcal{S}_I$ for some $I \in \mathcal{I}$, then it is unique and a solution exists for infinitely many values of $N$.  
\end{lem}
\begin{proof}
For any value of $I = (p,q,l)$, with $l \in \mathbb{Z}$ and $p,q \in \mathbb{Z}^{n}$, we can look for a solution of the form
$$ e^{2\pi i ( p \cdot x + q \cdot y + lt)} \in \mathcal{S}_I. $$
Substituting this into $\eqref{Kod dim eqns}$ and setting $A_{jk} = a_{j,k} + i b_{j,k}$, $C_j = c_j + i c_{j+n+1}$ we find that we have a solution if and only if the following is satisfied:
\begin{equation}\label{matrix S eqn}
2\pi \begin{pmatrix}
       -b_{1, 1} & \dots & -b_{1, 2n+1 }\\
    \vdots & & \vdots \\
    -b_{n+1, 1} & \dots & -b_{n, 2n+1}\\
    a_{1, 1} & \dots & a_{1, 2n+1 }\\
    \vdots & & \vdots \\
    a_{n+1, 1} & \dots & a_{n, 2n+1}\\
\end{pmatrix} \begin{pmatrix}
    p_1 \\ \vdots \\ p_{n} \\ q_1 \\ \vdots \\ q_{n}\\ l 
\end{pmatrix} = N \begin{pmatrix}
    c_1 \\ \vdots \\ c_{2n+2}
\end{pmatrix}.    
\end{equation}
First, we check for any real solutions $(p,q,l) \in \mathbb{R}^{2n+1}$ in the case when $N = 1$. Since the vectors $V_1, \dots, V_{n+1}, \Vbar_1, \dots, \Vbar_{n+1}$ are linearly independent we know that the matrix on the left hand side of \eqref{matrix S eqn} has maximal rank, therefore if a solution $(p,q,l)$ exists, it is unique. Furthermore, it implies that $(Np, Nq, Nl)$ is a solution for general $N$. 

If the solution is rational, \textit{i.e.}, contained in $\mathbb{Q}^{2n-1}$ then we can choose a value of $N$ such that $(Np, Nq, Nl)\in \mathbb{Z}^{2n+1}$ and we have a solution to \eqref{Kod dim eqns} for this $N$. In fact, there are infinitely many possible choices for this value of $N$.

If instead the solution is irrational or does not exist, then there are no solutions to \eqref{Kod dim eqns} for any choice of $N$.
\end{proof}
\begin{lem}\label{solns in T}
    There are no solutions to \eqref{Kod dim eqns} contained in $\mathcal{T}_J$ for any $J\in \mathcal{J}$.
\end{lem}
\begin{proof}
    Given any fixed left-invariant $J$, it is always possible to choose the vectors $V_1, \dots, V_{n+1}$ such that $A_{i 1}= \dots = A_{i n} = 0$.  
    The equation \eqref{Kod dim eqns}, when $j = 1$, can therefore be written as
\begin{align*}
&\Big[ A_{1\, n+1} \left(\diff{y_1} + x_1  \diff z \right) + \dots +  A_{1\, 2n} \left(\diff{y_n} + x_n  \diff z \right)+\\
 &+ A_{1\, 2n+1} \diff z  + A_{1\, 2n+2} \diff t  + NC_1\Big] f = 0. 
\end{align*}
    
    Let $f$ be a function in $\mathcal{T}_J$, \textit{i.e.}, $f$ is of the form
    $$e^{2\pi i (lt + q\cdot y + mz)}\sum_{\xi \in \Z^n}\psi(x + \xi) e^{2\pi i m \xi \cdot y} $$
    for some Schwartz function $\psi \in S(\mathbb R^n)$. Substituting this into the above equation, we see that if $f$ is a solution, then $\psi$ must satisfy
    $$B(\mathbf x) \psi(\mathbf x) = 0$$
    for all $\mathbf x = (x_1, \dots, x_n) \in \mathbb{R}^n$, where we define
    $$B(\mathbf x) := 2\pi i \left[A_{1\, n+1} (q_1 + mx_1) + \dots + A_{1\, 2n} (q_n + m x_n) + A_{1\, 2n+1} m + A_{1\, 2n+2}l\right] + NC_1.$$
    This implies we must have $B(\mathbf x) = 0$ for all $\mathbf x$ except when $\psi(\mathbf x) = 0$. 
    
    In the case when at least one of $A_{1\, n+1}, \dots, A_{1\, 2n}$ are non-zero, $B \neq 0$ on a dense subset of $\mathbb{R}^n$ and thus, by continuity, $\psi = 0$ everywhere, \textit{i.e.}, there are no non-trivial solutions.

    In the case when $A_{1\, n+1} = \dots = A_{1\, 2n} = 0$, we are instead left with
    $$ B = 2\pi i (A_{1\, 2n+1}m+ A_{1\, 2n+2} l) + NC_1. $$
    Note that $[V_j, \Vbar_1] = 0$ for all $j$. Thus, by the relationship between the Lie bracket and the exterior derivative, we conclude that $d\phi^j$ has no $\phi^{j\overline 1}$ component. Specifically, this means $d\phi^{12\dots n+1}$ has no $\phi^{12 \dots n+1} \wedge \overline{\phi}^1$ component and so $C_1 = 0$.
    Furthermore, since $V_1$ and $\Vbar_1$ must be linearly independent, $A_{1\, 2n+1}, A_{1\, 2n+2} \in \mathbb{C}$ cannot be real multiples of each other. All of this together means that, from $B = 0$, we can conclude that $l = m = 0$, but this contradicts the assumption that $m$ is non-zero. Therefore, $\mathcal{T}_J$ contains no solutions to $\eqref{Kod dim eqns}$. 
\end{proof}

Combining these two lemmas, we conclude that either $P_N = 1$ for an infinite number of choices of $N$ and therefore $\kappa_J = 0$, or $P_N = 0$ for all $N$ and therefore $\kappa_J = -\infty$.
\end{proof}

\section{Spectrum of the Laplacian}\label{spectrum}

In \cite{DS}, Denhinger and Singhof derive the spectrum of the Laplacian along with the corresponding eigenfunctions, on the Heisenberg manifold given by the quotient $H_{2n+1}(\Z) \backslash H_{2n+1}(\R)$. The description they gave made use of the Hermite polynomials.
\begin{defn}
    The Hermite functions are smooth maps $F_h:\R \rightarrow \R$
    defined for all $h \in \mathbb{N}_0$ by
    $$F_h(s) = (-1)^h e^{\frac{s^2} 2 } \frac{d^h}{ds^h}e^{-s^2}.$$
    These functions satisfy the following identities, which we shall make use of in Section \ref{spectrum}
   \begin{equation} \begin{split}\label{Hermite}
        F'_h(s) &= s F_h(s) -F_{h+1}(s),\\
        F'_h(s) &= 2hF_{h-1}(s) - s F_{h}(s),\\
        F''_h (s) &= s^2 F_h (s) - (2h+1) F_h (s).
    \end{split}
    \end{equation}
\end{defn}
The Kodaira-Thurston manifold can be written as the direct product of the Heisenberg manifold with a circle, and so the spectrum on $KT^{2n+2}$ can derived by an identical argument. 
\begin{thm}
    Consider the Laplacian acting on $C^{\infty}(KT^{2n+2})$ given by
    $$ a^2 \diffsq{t} +  b_i^2 \diffsq{x_i} +  c_i^2 \left(\diff{y_i} + x_i \diff{z} \right)^2 +  d^2 \diffsq{z},  $$
    for some positive choice of $ a,  d \in \mathbb{R}, b,c \in \R^{n}$. The eigenfunctions of this Laplacian are given by 
    $$f_{I} = f_{p,q,l} = e^{2\pi i (p\cdot x + q \cdot y+lt)} $$
    for any $I \in \mathcal{I}$ and
    $$g_{J,h}= g_{q,m,l,h} = e^{2\pi i (q\cdot y + mz+ lt)}\prod_{j=1}^n \sum_{\xi \in \Z} F_{h_j} \left(\sqrt{2\pi |m| \frac{ c_i} { b_i} } \left(x_j + \frac{q_j}{m}+\xi\right)\right) e^{2\pi i m \xi y_j} $$
    for any $J \in \mathcal{J}$ and any $h=(h_1,\dots,h_n) \in \N^n_0$.

    Specifically, we have
        $$\Delta_d f_{I} = -4\pi^2 \left( a^2 l^2 +  b_1^2  p_1^2 + \dots + b_n^2  p_n^2 + c_1^2 q_1^2 + \dots +  c_n^2 q_n^2 \right)f_{I} $$
        $$\Delta_d g_{J,h} = -\left[2\pi |m|(2 c_1^2 (h_1 + 1) + \dots + 2 c_n^2 (h_n + 1) ) + 4\pi^2 ( a^2 l^2 +  d^2 m^2)\right]g_{J,h}.   $$
\end{thm}

Note that the space $\mathcal{S}_I$ contains only complex multiples of the function $f_I$, while the space $\mathcal{T}_J$ contains the functions $g_{J,h}$ for all $h \in \mathbb{N}$ and we have the following result.

\begin{cor}\label{basis of T}
    Given any fixed index $J \in \mathcal{J}$, the functions $(g_{J,h})_{h \in \N_0^n}$ form an orthogonal basis of $\mathcal{T}_{J}$. 
    \begin{proof}
        It is well-known that on a compact manifold the eigenfunctions of the Laplacian, or indeed any self adjoint elliptic operator, form an orthogonal basis of $L^2(M)$ (see \cite[Theorem 14, Ch. XI]{Pal}). The corollary then follows from the fact that $\mathcal{T}_J$ is orthogonal to $f_{I}$ and $g_{J',h}$ whenever $J' \neq J$, but contains $g_{J',h}$ if $J' = J$.
        \end{proof}
\end{cor}

\begin{prop}\label{W relations}
    On the manifold $KT^{2n+2}$ define the following left-invariant frame on the tangent bundle
    $$U_j =  b_j \diff{x_j} + i  c_j \left[ \diff{y_j} + x_j \diff{z}\right], \quad \quad U_{n+1} =  a \diff{t} + i d\diff{z},$$
    where $j = 1, \dots, n$.

    The functions $g_{J,h}$ satisfy the following relations, for all $I \in \mathcal{I}$, $h = (h_1, \dots, h_n) \in \mathbb{N}_0^{n}$.
    $$U_{n+1} g_{J,h} = 2\pi i \left( a l + i d m \right) g_{J,h}, \quad \quad \overline{U_{n+1}} g_{J,h} = 2\pi i \left(  a l - i  d m \right) g_{J,h}. $$
    If $m > 0$ then, for $j = 1, \dots, n$,
    $$ U_j g_{J,h} = -\sqrt{2\pi m b_j c_j} \,g_{J,h+e_j}, \quad \quad \overline{U_j} g_{J,h} = 2h_j \sqrt{2\pi m b_j c_j}\,g_{I, h-e_j}.  $$
    If $m<0$ then, for $j = 1, \dots, n$,
    $$U_j g_{J,h} = 2h_j \sqrt{-2\pi m b_j c_j}\,g_{J,h-e_j}, \quad \quad \overline{U_j} g_{J,h} = - \sqrt{-2\pi m b_j c_j} \, g_{J,h+e_j}.$$
    Here we use $e_j$ to denote the element of $\mathbb{N}_0^n$ with a 1 in the $jth$ position and zeros in all other positions, \textit{i.e.}, $h+e_j = (h_1, \dots, h_j + 1, \dots, h_n)$.
\begin{proof}
    The relations for $U_{n+1}$ and $\overline{U_{n+1}}$ are clear from the definition of $g_{J,h}$. To prove the relations for $U_{j}$ and $\overline{U_{j}}$ with $j = 1, \dots, n$, we shall assume $m > 0$. The case when $m < 0$ follows from an identical argument.

    For simplicity of notation, we shall substitute $s_j = \sqrt{2\pi m \frac{c_j}{b_j}}(x_j + \xi + \frac{q_j}{m})$, so that
    $$ g_{J,h} = e^{2\pi i ( q\cdot y + mz+lt)} \prod_{i = 1}^{n} \sum_{\xi \in \Z} F_{h_i}(s_i)e^{2\pi i m \xi y_i}.$$

    Applying $U_j$ to $g_{J,h}$ and writing $\diff{x_j} = \sqrt{2\pi m \frac{ c_j}{ b_j}} \diff{s_j}$, we see that
    \begin{align*}
        \begin{split} 
        U_{j}g_{J,h} &= e^{2\pi i( q\cdot y + mz +lt)} 
        \sum_{\xi \in \Z}
                \sqrt{2\pi m  b_j  c_j}\left[\frac{d}{ds_j}F_{h_j}(s_j) - s_j F_{h_j}(s_j)\right]e^{2\pi i m\xi y_j}\\
        & \quad \quad \quad \quad\quad \quad \quad\cdot 
        \prod_{i \neq j}
           \left( \sum_{\xi \in \Z}
                 F_{h_i}(s_i)e^{2\pi i m \xi y_i}\right)\\
        &= -\sqrt{2\pi m b_j c_j} \, e^{2\pi i( q\cdot y + mz+lt)} 
        \sum_{\xi \in \Z}
                 F_{h_j +1}(s_j) e^{2\pi i m\xi y_j}\\
        & \quad \quad \quad \quad\quad \quad \quad \cdot 
        \prod_{i \neq j}
           \left( \sum_{\xi \in \Z}
                 F_{h_i}(s_i)e^{2\pi i m \xi y_i}\right) = - \sqrt{2\pi m b_j c_j} \, g_{J,h+e_j}.\\
        \end{split}
    \end{align*} 
    The second equality follows from the identities \eqref{Hermite}.
    Similarly, by applying $\overline{U_j}$ to $g_{J,h}$ we see that 
    \begin{align*}
        \begin{split} 
        \overline{U_{j}}g_{J,h} &= e^{2\pi i( q\cdot y + mz+lt)} 
        \sum_{\xi \in \Z}
                \sqrt{2\pi m  b_j  c_j}\left[\frac{d}{ds_j}F_{h_j}(s_j) + s_j F_{h_j}(s_j)\right]e^{2\pi i m\xi y_j}\\
        & \quad \quad \quad \quad\quad \quad \quad\cdot 
        \prod_{i \neq j}
           \left( \sum_{\xi \in \Z}
                 F_{h_i}(s_i)e^{2\pi i m \xi y_i}\right)\\
        &= 2 h_j \sqrt{2\pi m  b_j  c_j} \, e^{2\pi i(q\cdot y + mz+lt)} 
        \sum_{\xi \in \Z}
                 F_{h_j -1}(s_j) e^{2\pi i m\xi y_j}\\
        & \quad \quad \quad \quad\quad \quad \quad \cdot 
        \prod_{i \neq j}
           \left( \sum_{\xi \in \Z}
                 F_{h_i}(s_i)e^{2\pi i m \xi y_i}\right) = 2 h_j \sqrt{2\pi m b_j c_j} \, g_{J,h-e_j}.\qedhere
        \end{split}
    \end{align*}    
\end{proof}
    \end{prop}

\section{Example: $\delbar$-harmonic $(0,1)$-forms on $KT^4$}

We will now consider an example, calculating $h^{0,1}_{\bar \partial}$ on $KT^4$ for some family of almost Hermitian structures.
Define an almost complex structure $J = J_{\beta,\delta}$  on $KT^4$, depending on $\beta, \delta \in \R$, given by
    $$J:\diff{t}\mapsto \beta\diff{x}, \quad \quad J: \diff{y}+ x \diff{z} \mapsto \delta \diff{z}.$$
A Hermitian metric can then be chosen such that
$$\diff t,\ \beta \diff x,\ \diff y +x \diff z,\ \delta \diff z$$
are orthonormal.

The corresponding Laplacian is given by
 $$ \diffsq{t} + \beta^2 \diffsq{x} +  \left(\diff{y} + x \diff{z} \right)^2 + \delta^2 \diffsq{z},  $$
 and so we define
 $$U_1 = \beta \diff{x} + i  \left[ \diff{y} + x \diff{z}\right], \quad \quad U_2 = \diff{t} + i\delta \diff{z}.$$

A frame for the vector bundle $T^{1,0}KT^4$ can be given by
$$ V_1 = \frac 12 \left( \diff t -i \beta \diff x \right), \quad \quad V_2 = \frac 12 \left( \left[ \diff y + x \diff z \right] - i \delta\diff z \right)$$
along with its dual frame
$$\phi^1 = dt +  \frac i \beta dx, \quad \quad \phi^2 = dy + \frac i \delta \left[ dz-xdy\right],$$
which satisfies the structure equations
$$d\phi^1 = 0, \quad d\phi^2 = -\frac{\beta}{4\delta} \left(\phi^{12} + \phi^{1 \bar 2} + \phi^{2 \bar 1} - \phi^{\bar 1 \bar 2}\right).$$

Denoting a general $(0,1)$-form by
$$s = f \overline{\phi^1} + g \overline{\phi^2}$$
we see that $s$ is $\overline \partial$-harmonic iff $\overline \partial s = \overline{\partial}^* = 0$, iff
$$\begin{cases}
    -\overline{V_2} f + \overline{V_1} g - \frac{\beta}{4\delta}g = 0,\\
    V_1 f + V_2 g = 0.
\end{cases}$$
Rewriting this using $U_1, U_2, \overline{U_1}$ and $\overline{U_2}$ gives us
\begin{equation}\label{KT4 W eqn}
    \begin{cases}
    -\left(\left(U_2 - \overline{U_2}\right) - i\left(U_1 - \overline{U_1}  \right)\right)f + \left(\left(U_2 + \overline{U_2} \right) + i \left(U_1 + \overline{U_1} \right) \right) g - \frac{\beta}{\delta}g = 0,\\
    \left( \left( U_2 + \overline{U_2} \right) -i \left(U_1 + \overline{U_1}\right) \right) f - \left( \left( U_2 - \overline{U_2} \right) +i \left(U_1 - \overline{U_1}\right) \right) g = 0.
\end{cases}
\end{equation}

By taking either the sum or the difference of the two equations in $\eqref{KT4 W eqn}$ we construct a new system of equations

\begin{equation*}
    \begin{cases}
    \left( \overline{U_2} - i \overline{U_1} \right)f + \left(\overline{U_2} + i \overline{U_1}  \right) g - \frac{\beta}{\delta}g = 0,\\
    \left(- U_2 +i U_1 \right) f + \left( U_2+i U_1 \right) g -  \frac{\beta}{\delta}g = 0.
\end{cases}
\end{equation*}

Then, defining functions $F,G \in C^{\infty}(KT^4)$ such that
$$f = \frac{F + G}{2}, \quad  g = \frac{F - G}{2},$$
the system of equations becomes
\begin{equation*}
    \begin{cases}
    \left( \overline{U_2} -  \frac{\beta}{\delta} \right) F  - \left(i \overline{U_1} - \frac{\beta}{\delta}\right)G  = 0,\\
    \left(i U_1 - \frac{\beta}{\delta}\right) F - \left( U_2 - \frac{\beta}{\delta} \right) G = 0.
\end{cases}
\end{equation*}

As a consequence of Corollary \ref{Cor decomp of solns}, along with the decomposition \eqref{decomp of KT6}, it is sufficient to find the solutions $F,G$ which are contained within the spaces $\mathcal{S}_I$ and $\mathcal{T}_J$.

A computation of the solutions in $\mathcal{S}_I$ can be found in \cite[Section 3.3]{HZ}. In the next part of this section, we will focus on finding the solutions in $\mathcal{T}_J$, using a simpler method than \cite[Section 3.2]{HZ}.

\subsection{Solutions in $\mathcal{T}_J$}
From Corollary \ref{basis of T} we know that the functions $g_{J,h}$ form a basis of $\mathcal{T}_J$ and so we can write
$$F = \sum_{h= 0}^\infty A_h g_{J,h}, \quad \quad G = \sum_{h= 0}^\infty B_h g_{J,h}$$
where $A_h, B_h \in \mathbb{C}$ are sequences of complex numbers. Then, using the relations described in Proposition \ref{W relations}, the above system of equations yields a pair of recurrence relations on $A_h$ and $B_h$
$$\left(2\pi i (l-i\delta m) - \frac{\beta}{\delta}\right)A_h + \frac{\beta}{\delta} B_h - 2i (h+1) \sqrt{2\pi m \beta} B_{h+1} = 0,$$
$$ \left(-2\pi i (l+i\delta m) + \frac{\beta}{\delta}\right)B_h - \frac{\beta}{\delta} A_h - i \sqrt{2\pi m \beta} A_{h-1} = 0. $$
In principle, it should be possible to describe the asymptotic behaviour of this system as $h \rightarrow \infty$. This would allow us to determine the sequences $(A_h, B_h)_{h\in \mathbb{N}_0}$ for which the sum over $A_h g_{J,h}$ and $B_h g_{J,h}$ converges.

In this example however, there is a simpler way to find the solutions. If, instead of 
$$ g_{J,h} = e^{2\pi i ( qy + mz+ lt)} \sum_{\xi \in \Z} F_{h} \left(\sqrt{2\pi |m| \frac{1} {\beta} } \left(x + \frac{q}{m}+\xi\right)\right) e^{2\pi i m \xi y},$$
we use a slightly modified basis
$$\widetilde{g}_{J,h}  = e^{2\pi i ( qy + mz+lt)} \sum_{\xi \in \Z}e^{-i\frac{x}{\delta}} F_{h} \left(\sqrt{2\pi |m| \frac{1} {\beta} } \left(x + \frac{q}{m}+\xi\right)\right) e^{2\pi i m \xi y},$$
then we have (for $m > 0$)
$$U_1 \widetilde{g}_{J,h} = -\sqrt{2\pi m \beta} \widetilde{g}_{J,h+1} - i\frac{\beta}{\delta} \widetilde{g}_{J,h},$$
$$\overline{U_1} \widetilde{g}_{J,h} = 2 h \sqrt{2\pi m \beta} \widetilde{g}_{J,h-1} - i\frac{\beta}{\delta} \widetilde{g}_{J,h}.$$
By writing 
$$F = \sum_{h= 0}^\infty \widetilde{A}_h \widetilde{g}_{J,h}, \quad \quad G = \sum_{h= 0}^\infty \widetilde{B}_h \widetilde{g}_{J,h}$$
we obtain a new recurrence relation for $\widetilde{A}_h$ and $\widetilde{B}_h$
$$\left(2\pi i (l-i\delta m) - \frac{\beta}{\delta}\right)\widetilde{A}_h - 2i (h+1) \sqrt{2\pi m \beta} \widetilde{B}_{h+1} = 0,$$
$$ \left(-2\pi i (l+i\delta m) + \frac{\beta}{\delta}\right)\widetilde{B}_h - i \sqrt{2\pi m \beta} \widetilde{A}_{h-1} = 0. $$
Eliminating the $A_h$ terms, we get
$$ \left(4\pi^2(l^2 + \delta^2 m^2) + 4\pi i l \frac{\beta}{\delta} - \frac{\beta^2}{\delta^2} + 4\pi m h \beta \right) B_h = 0, $$
thus we have a non-trivial solution iff
$$4\pi^2(l^2 + \delta^2 m^2) + 4\pi i l \frac{\beta}{\delta} - \frac{\beta^2}{\delta^2} + 4\pi m h \beta  \neq 0.$$
Considering the imaginary part first, we find that $l = 0$. The real part then gives us the condition
$$ 4\pi^2 \delta^4 m^2 + 4 \pi m h \beta \delta^2  - \beta^2\neq 0 $$
In this way we have re-obtained the same result of \cite[Section 3.2]{HZ} with a different method.

\section{Example: $\delbar$-harmonic $(0,1)$-forms on $KT^6$}

We will now consider an example, calculating $h^{0,1}_{\bar \partial}$ on $KT^6$ for some family of almost Hermitian structures.
Define an almost complex structure $J_{a,b,c}$  on $KT^6$, depending on $a,b,c \in \R\setminus\{0\}$, given by
\begin{align*}
    &\diff{x_1}\mapsto a\diff{x_2}\\
    J: \quad &\diff{y_1}+ x_1 \diff{z} \mapsto b\left(\diff{y_2}+ x_2 \diff{z}\right)\\
    &\diff{t} \mapsto c\diff{z}.
\end{align*}
Any left invariant Hermitian metric $g_{a,b,c,\rho,\sigma,\tau}$ can then be chosen by setting
\[
\rho\diff{x_1},\ \rho a\diff{x_2},\ \sigma\left(\diff{y_1}+ x_1 \diff{z}\right),\ \sigma b\left(\diff{y_2}+ x_2 \diff{z}\right),\ \tau\diff{t},\ \tau c\diff{z}
\]
to be orthonormal vectors, for $\rho,\sigma,\tau\in\R\setminus\{0\}$.

It will be convenient to define the vector fields
 \begin{align*}
U_1 &=  \diff{x_1} + i  \left[ \diff{y_1} + x_1 \diff{z}\right], \quad \quad U_3 = \diff{t} + ci \diff{z},\\
U_2 &=  a\diff{x_2} + bi  \left[ \diff{y_2} + x_2 \diff{z}\right],
\end{align*}
while frame of $(1,0)$-vector fields is
\begin{align*}
V_1 &= \frac 12 \left( \diff{x_1} -ai \diff{x_2} \right),\quad \quad  V_3 = \frac 12 \left(\diff{t} -ci  \diff{z} \right),\\
V_2 &= \frac 12 \left( \left[\diff{y_1}+ x_1 \diff{z} \right]-bi \left[\diff{y_2} + x_2 \diff z \right]\right),
\end{align*}
and the dual frame of $(1,0)$-forms is then given by
$$ \phi^1 = dx_1 + \frac i a dx_2, \quad \quad \phi^2 = dy_1 + \frac i b dy_2, \quad \quad \phi^3 = dt +\frac i c \left[dz -x_1 dy_1 - x_2 dy_2\right] ,  $$
with structure equations
$$d\phi^1 = 0, \quad \quad d\phi^2 = 0, $$
$$ d\phi^3 = i\frac {ab -1}{4c} \left(\phi^{12}+ \phi^{\bar 1 \bar 2}\right) + i\frac{ab + 1}{4c} \left(\phi^{2 \bar 1} - \phi^{1 \bar 2} \right).  $$

A general $(0,1)$-form can be written as  $s = F\overline{ \phi^1} + G \overline{\phi^2} + K \overline{\phi^3} \in \mathcal{A}^{0,1}(KT^6)$, for some smooth complex valued functions $F,G,K \in C^{\infty}(KT^6)$.  We want to know when $s \in \mathcal{H}^{0,1}_{\bar \partial}$, \textit{i.e.}, when $\Delta_{\bar\partial}s = 0$, which holds if and only if the two conditions $\bar\partial s = 0 $ and $ \bar \partial^* s = 0$ are satisfied. From these we obtain the PDEs
\begin{equation}\label{system}
    \begin{cases}
-\overline{V_2}F + \overline{V_1}G + i\frac{ab -1 }{4c}K = 0,\\
-\overline{V_3}F + \overline{V_1}K = 0,\\
-\overline{V_3}G + \overline{V_2}K = 0,\\
\sigma^2\tau^2 V_1 F + \rho^2\tau^2 V_2 G + \rho^2\sigma^2 V_3 K = 0.
\end{cases}
\end{equation}

Notice that these PDEs all involve only left-invariant differential operators and so the decomposition of functions given in \eqref{decomp of KT6} can help us find solutions.

In particular, it is sufficient to find the solutions when $F,G,K$ are in $\mathcal{S}_{I}$ for fixed $I \in \mathcal{I}$ and the solutions when $F,G,K$ are in $\mathcal{T}_{J}$ for fixed $J \in \mathcal{J}$.

\subsection{Solutions in $\mathcal{S}_{I}$}
In this section we will make use of the basis of $\mathcal{S}_{I}$ given by 
$$f_{I} = f_{p,q,l} = e^{2\pi i (p_1 x_1 + p_2 x_2 + q_1 y_1 + q_2 y_2 +lt)} $$
for any $I = (p_1,p_2,q_1,q_2,l) \in \mathcal{I}=\Z^5$. The functions $F,G,K$ can be expressed in this basis as
\[
F = \sum_{I\in \mathcal{I}} F_{I} f_{I},\quad\quad  G = \sum_{I\in \mathcal{I}} G_{I} f_{I},\quad\quad  K = \sum_{I\in \mathcal{I}} K_{I} f_{I}.
\]
The vector fields $V_1,V_2,V_3,\c{V_1},\c{V_2},\c{V_3}$ operate on $f_I$ as
\begin{align*}
V_1f_I&=\pi i(p_1-iap_2)f_I, & V_2f_I&=\pi i(q_1-ibq_2)f_I, & V_3f_I&=\pi i lf_I,\\
\c{V_1}f_I&=\pi i(p_1+iap_2)f_I, & \c{V_2}f_I&=\pi i(q_1+ibq_2)f_I, & \c{V_3}f_I&=\pi i lf_I.
\end{align*}
Therefore system \eqref{system} traduces into the following system for the coefficients of $F,G,K$
\begin{equation*}
\begin{cases}
-\pi (q_1+ibq_2)F_I+\pi (p_1+iap_2)G_I+\frac{ab-1}{4c}K_I=0,\\
-lF_I+(p_1+iap_2)K_I=0,\\
-lG_I+(q_1+ibq_2)K_I=0,\\
\sigma^2\tau^2(p_1-iap_2)F_I+\rho^2\tau^2(q_1-ibq_2)G_I+\rho^2\sigma^2lK_I=0.
\end{cases}
\end{equation*}
If $l=0$, then from the second and third equations either $K\in\C$ or $I=0$. In both cases from the first equation either $K=0$ or $ab=1$, \textit{i.e.}, the almost complex structure is integrable.
If $l\ne0$, one can substitute $F_I$ and $G_I$ from the second and the third equations into the first one, obtaining either $K=0$ or $ab=1$.

Thus, if $K=0$, from the first and the last equations we deduce $F,G\in\C$. On the other hand, if $ab=1$, again from the first and the last equations it follows that $F,G,K\in\C$.

\subsection{Solutions in $\mathcal{T}_{J}$}

In this section we will make use of the basis of $\mathcal{T}_{J}$ given by 
$$g_{J,h} = e^{2\pi i( q\cdot y + m z+lt)} \prod_{j=1,2}\sum_{\xi \in \Z}F_{h_j}(\sqrt{2\pi |m|}(x_j + \frac{q_j}{m} + \xi))e^{2\pi i m \xi y_j}$$ 
in Corollary \ref{basis of T}. It will be sufficient to consider the case $\rho=\sigma=\tau=1$.

Transform the system \eqref{system} as follows: 
\begin{itemize}
    \item Make the substitution $F = \alpha + \beta$, $G = i(-\alpha + \beta)$, where $\alpha,\beta \in \mathcal{T}_{J}$.
    \item Replace the second and third equations with the sum of the second equation with $\pm i$ times the third.
\end{itemize}
This yields the system
\begin{equation*}
    \begin{cases}
        (-i\overline{V_1}-\overline{V_2})\alpha + (i\overline{V_1} - \overline{V_2})\beta + i\frac{ab - 1}{4c} K =0,\\
        -2 \overline{V_3} \alpha + (\overline{V_1} + i \overline{V_2})K = 0, \\
         -2\overline{V_3}\beta + (\overline{V_1}-i \overline{V_2})K=0, \\
        (V_1 - i V_2)\alpha + (V_1 + i V_2)\beta + V_3 K = 0.
    \end{cases}
\end{equation*}

Taking into account that
\[
2iV_1-2V_2=iU_1+U_2,\quad\quad 2iV_1+2V_2=i\c{U_1}+\c{U_2}, \quad\quad 2\overline{V_3}=U_3,
\]
this system is equivalent to
\begin{equation*}
    \begin{cases}
  (-i\c{U_1} + \c{U_2})\alpha +  (i{U_1} - {U_2})\beta + i\frac{ab-1}{2c} K =0,\\
        2iU_3 \alpha + (-i{U_1} + {U_2})K = 0, \\
        2iU_3\beta + (-i\c{U_1} + \c{U_2})K=0, \\
        (i\c{U_1}+\c{U_2})\alpha + (iU_1+U_2)\beta + i\c{U_3} K = 0.
    \end{cases}
\end{equation*}

Now we make use of the basis given by $g_{J,h}$ to write $\alpha,\beta,K$ as
$$\alpha = \sum_{h_1,h_2 \in \N_0} A_{h_1,h_2} g_{h_1,h_2},  $$
$$\beta = \sum_{h_1,h_2 \in \N_0} B_{h_1,h_2} g_{h_1,h_2},  $$
$$K = \sum_{h_1,h_2 \in \N_0} K_{h_1,h_2} g_{h_1,h_2}.  $$

Substituting this into the above system of equations and assuming that $m>0$ gives conditions on the coefficients $A_{h_1,h_2}, B_{h_1,h_2}, K_{h_1,h_2} \in \C$.

\begin{equation*}
    \begin{cases}
2\sqrt{2\pi m}\left[-i(h_1+1)A_{h_1+1,h_2}+\sqrt{ab}(h_2+1)A_{h_1,h_2+1}\right]+\\
+\sqrt{2\pi m}\left[-iB_{h_1-1,h_2}+\sqrt{ab}B_{h_1,h_2-1}\right]+ i\frac{ab-1}{2c} K_{h_1,h_2}=0,\\ \\

-4\pi(l+icm)A_{h_1,h_2}+\sqrt{2\pi m}\left[iK_{h_1-1,h_2}-\sqrt{ab}K_{h_1,h_2-1}\right]=0,\\ \\

-4\pi(l+icm)B_{h_1,h_2}+2\sqrt{2\pi m}\left[-i(h_1+1)K_{h_1+1,h_2}+\sqrt{ab}(h_2+1)K_{h_1,h_2+1}\right]=0,\\ \\

2\sqrt{2\pi m}\left[i(h_1+1)A_{h_1+1,h_2}+\sqrt{ab}(h_2+1)A_{h_1,h_2+1}\right]+\\
-\sqrt{2\pi m}\left[iB_{h_1-1,h_2}+\sqrt{ab}B_{h_1,h_2-1}\right]-2\pi(l-icm)K_{h_1,h_2}=0.
    \end{cases}
\end{equation*}

Substituting the coefficients of $\alpha$ and $\beta$ from the second and third equations into the first and the last equations, we obtain 
\begin{equation*}
    \begin{cases}
K_{h_1,h_2}(1-ab)(3m-i\frac{l}{c})=0,\\ \\

K_{h_1,h_2}\left[2\pi(l^2+c^2m^2)+m(2h_1+1)\right]=0.
    \end{cases}
\end{equation*}
We see that independently from both equations $K_{h_1,h_2}=0$ for all $h_1,h_2\in\N_0$. In particular, since the first equation in the system is obtained just from asking $\delbar s=0$, we deduce that if $ab\ne1$ then there are no (non zero) $\delbar$ closed $(0,1)$-forms lying in $\mathcal{T}_J$.
A nearly identical argument shows that the same is true when $m<0$.

\begin{cor}
If $ab\ne1$, then for any choice of a left invariant metric $g$ on $(KT^6,J_{a,b,c})$ we have
\[
\mathcal{H}^{0,1}_{\delbar}=\C<\c{\phi^1},\c{\phi^2}>.
\]
If $ab=1$, \textit{i.e.}, $J_{a,b,c}$ is integrable, then on $(KT^6,J_{a,b,c})$ we have
\[
\mathcal{H}^{0,1}_{\delbar}\simeq H^{0,1}_{\delbar}\simeq \C<\c{\phi^1},\c{\phi^2},\c{\phi^3}>.
\]
\end{cor}
\begin{proof}
If $ab\ne1$, then the solutions of $\Delta_{\delbar}s=0$ in $\mathcal{S}_{I}$ are generated by $\c{\phi^1},\c{\phi^2}$ for any choice of a left invariant metric $g_{a,b,c,\rho,\sigma,\tau}$, while in $\mathcal{T}_J$ there are no (non zero) $\delbar$ closed $(0,1)$-forms and thus no $\delbar$ harmonic $(0,1)$-forms. If $ab=1$, then Dolbeault cohomology is metric independent and so we can choose the metric $g_{a,b,c,1,1,1}$ to compute it solving $\Delta_{\delbar}s=0$. The solutions in $\mathcal{S}_{I}$ are generated by $\c{\phi^1},\c{\phi^2},\c{\phi^3}$, while in $\mathcal{T}_J$ there are no (non zero) $\delbar$ harmonic $(0,1)$-forms.
\end{proof}

\end{document}